\newif\ifHAL
\HALtrue

\ifHAL
\documentclass[10pt,a4paper]{article}
\else
\documentclass[10pt]{m2an}

\fi

\usepackage[colorlinks,citecolor=cyan,linkcolor=magenta]{hyperref}
\usepackage[latin1]{inputenc}
\usepackage{amsmath}
\usepackage{amsmath,bm}
\usepackage{amsfonts}
\usepackage{amssymb}
\usepackage{subcaption}
\usepackage{graphicx}
\usepackage{xspace}
\usepackage{tikz}
\usepackage{fullpage}

\usepackage[colorinlistoftodos]{todonotes}

\ifHAL
\usepackage{theorem}
\newtheorem{corollary}{Corollary}[section]
\newtheorem{lemma}[corollary]{Lemma}
\newtheorem{theorem}[corollary]{Theorem}
\newtheorem{proposition}[corollary]{Proposition}
\newtheorem{definition}[corollary]{Definition}

\newtheorem{remark}[corollary]{Remark}

\newtheorem{assumption}[corollary]{Assumption}

\newcommand{\qed}{ \vspace{-.5cm}\hfill $\Box$ }
\newenvironment{proof}[1][Proof.]{\begin{trivlist}
		\item[\hskip \labelsep {\bfseries #1}]}{\end{trivlist}\qed}
\else
\theoremstyle{plain}
\newtheorem{theorem}{Theorem}[section]
\newtheorem{remark}[theorem]{Remark}
\newtheorem{lemma}[theorem]{Lemma}
\newtheorem{corollary}[theorem]{Corollary}

\fi

\def\bi{\begin{itemize}}
\def\ei{\end{itemize}}
\def\berom{\begin{enumerate}[{\rm(i)}]}
\def\eerom{\end{enumerate}}

\newcounter{itemrem}

\def\eRem{\qed\end{Rem}}
\makeatletter
\def\bRem{\@ifnextchar[{\@remwithtitle}{\@remwithouttitle}}
\def\@remwithtitle[#1]{\begin{Rem}[#1]\setcounter{itemrem}{0}}%
\def\@remwithouttitle{\begin{Rem}\setcounter{itemrem}{0}}%
\makeatother

\newcounter{itemexp}

\def\eExp{\qed\end{Exp}}
\makeatletter
\def\bExp{\@ifnextchar[{\@expwithtitle}{\@expwithouttitle}}
\def\@expwithtitle[#1]{\begin{Exp}[#1]\setcounter{itemexp}{0}}%
\def\@expwithouttitle{\begin{Exp}\setcounter{itemexp}{0}}%
\makeatother

\def\bproof{\begin{proof}}
\def\eproof{\qed\end{proof}}

\def\bExo{\begin{Exo}}
\def\eExo{\end{Exo}}

\def\wS{{\widehat S}}

\newcommand{\bfjlg}{\bm}         % Put the right def to get back bold symbols
\newcommand{\boldsymboljlg}{\boldsymbol} % Put the right def to get back bold symbols

\newcommand{\ba}{{\bfjlg a}}
\newcommand{\bb}{{\bfjlg b}}

\newcommand{\bd}{{\bfjlg d}}

\newcommand{\bef}{{\bfjlg f}}

\newcommand{\bl}{{\bfjlg l}}

\newcommand{\bn}{{\bfjlg n}}

\newcommand{\bv}{{\bfjlg v}}
\newcommand{\bw}{{\bfjlg w}}
\newcommand{\bx}{{\bfjlg x}}

\newcommand{\bA}{{\bfjlg A}}

\newcommand{\bC}{{\bfjlg C}}

\newcommand{\bH}{{\bfjlg H}}

\newcommand{\bL}{{\bfjlg L}}

\newcommand{\bP}{{\bfjlg P}}

\newcommand{\bT}{{\bfjlg T}}

\newcommand{\bV}{{\bfjlg V}}
\newcommand{\bW}{{\bfjlg W}}

\newcommand{\bY}{{\bfjlg Y}}

\newcommand{\bxi}{{\boldsymboljlg \xi}}

\newcommand{\bphi}{{\boldsymboljlg \phi}}

\newcommand{\bzero}{\bfjlg{0}}
\newcommand{\calF}{{\mathcal F}}

\newcommand{\calH}{{\mathcal H}}

\newcommand{\calJ}{{\mathcal J}}

\newcommand{\calT}{{\mathcal T}}

\newcommand{\polN}{{\mathbb N}}

\newcommand{\bpolN}{\pmb{\mathbb N}}

\newcommand{\bpolP}{\pmb{\mathbb P}}

\newcommand{\upc}{^{\mathrm{c}}}
\newcommand{\upd}{^{\mathrm{d}}}

\newcommand{\upcb}{^{\mathrm{c,b}}}

\newcommand{\upcav}{^{\mathrm{c,av}}}
\newcommand{\updav}{^{\mathrm{d,av}}}

\newcommand{\interior}{\mathop{\mbox{\rm int}}}

\def\Real{{\mathbb R}}

\def\inter{\mathcal{I}}

\def\lb{[\![}
\def\rb{]\!]}

\newcommand{\upint}{^\circ}
\newcommand{\upbnd}{^\partial}

\def\hinH{{h\in\calH}}
\def\subhinH{_\hinH}
\def\famTh{(\calT_h)\subhinH}

\def\calFh{\calF_h}
\def\calFK{\calF_K}
\def\calFhi{\calFh\upint}
\def\calFhb{\calFh\upbnd}

\newenvironment{Ventry}[1]%
{\begin{list}{}{%
\settowidth{\labelwidth}{{\rm (#1)}}%
\setlength{\leftmargin}{\labelwidth+\labelsep}}}
{\end{list}}

\newcommand{\Dom}{D}

\newcommand{\intset}[2]{\{#1\hskip.05em\relax{:}\hskip.1em\relax#2\}}
\newcommand{\term}{\mathfrak{T}}

\newcommand{\loS}{_{\mathrm{\scriptscriptstyle S}}}

\newcommand{\eqq}{\mathrel{\mathop:}=}
\newcommand{\qqe}{=\mathrel{\mathop:}}

\def\DIV{\nabla{\cdot}}      
\def\ROT{\nabla{\times}} 
\def\GRAD{\nabla}

\def\SCAL{{\cdot}}

\def\dif{\,\mathrm{d}}

\def\front{{\partial\Dom}}

\def\Ldeuxd{{\bL^2(\Dom)}}

\def\Hdiv{{{\bH(\text{\rm div};\Dom)}}}

\def\Hrot{{{\bH(\text{\rm curl};\Dom)}}}
\def\Hrotz{{{\bH_0(\text{\rm curl};\Dom)}}}
\def\jump#1{\lb{#1}\rb}

\newcommand{\eg}{e.g.,\@\xspace}

\newcommand{\bVentry}[1]{\begin{Ventry}{#1}}
\newcommand{\eVentry}{\end{Ventry}}

\catcode`\@=11
\newdimen\linespacing
\linespacing=\baselineskip
\newcommand\subsubsectionmodif{\@startsection{subsubsection}{3}%
  {\z@}{.5\linespacing\@plus.7\linespacing}{-.5em}{\bfseries}*}

\newlength\pageone
\newlength\pagetwo
\newlength\interspace

\newlength\retraitspace
\newlength\calculspace

\newcommand{\tq}{{\;|\;}} 
\newcommand{\st}{\tq}

\DeclareMathOperator*{\essinf}{ess\,inf}



\makeatletter
\newcommand{\manuallabel}[2]{{\textup{(#2)}\def\@currentlabel{#2}\label{#1}}}
\makeatother

\newcounter{subeq}

\newcommand{\ale}[1]{{#1}}

\begin{document}
	
	\title{Local decay rates of best-approximation errors using vector-valued finite elements for fields with low regularity and integrable curl or divergence}
	
	\ifHAL
	\author{
		Zhaonan Dong\thanks{
			Inria Paris, 75589 Paris, France,
			and CERMICS, Ecole des Ponts, 77455 Marne-la-Vall\'{e}e cedex 2, France.
			{\tt{zhaonan.dong@inria.fr}}.}
		\and Alexandre Ern\thanks{
			CERMICS, Ecole des Ponts, 77455 Marne-la-Vall\'{e}e cedex 2, France,
			and  Inria Paris, 75589 Paris, France.
			{\tt{alexandre.ern@enpc.fr}}.}
		\and Jean-Luc Guermond\thanks{
			Department of Mathematics, Texas A\&M University, 3368 TAMU, College Station, TX 77843, USA.
			{\tt{guermond@math.tamu.edu}}.}
	}
	\else
	\author{Zhaonan Dong}
	\address{Inria Paris, 75589 Paris, France.
		\email{zhaonan.dong@inria.fr}}
	\secondaddress{CERMICS, Ecole des Ponts, 77455 Marne-la-Vall\'{e}e cedex 2, France. \email{alexandre.ern@enpc.fr} }
	\author{Alexandre Ern}
	\sameaddress{2, 1}
	\author{Jean-Luc Guermond}
	\address{ Department of Mathematics, Texas A\&M University, 3368 TAMU, College Station, TX 77843, USA. \email{guermond@math.tamu.edu}}
	\fi
	
	\date{\today}
	
	\ifHAL
	\else
	\begin{abstract}
		We estimate best-approximation errors using vector-valued finite 
		elements for fields with low regularity in the scale of fractional-order
		Sobolev spaces. By assuming additionally 
		that the target field has a curl or divergence property, we establish 
		upper bounds on these errors that can be localized to the mesh cells. 
		These bounds are derived using the quasi-interpolation errors with or
		without boundary prescription derived in [A. Ern and J.-L. Guermond,
		ESAIM Math. Model. Numer. Anal., 51 (2017), pp.~1367--1385]. 
		By using the face-to-cell lifting operators analyzed in 
		[A. Ern and J.-L. Guermond, Found. Comput. Math., (2021)],
		and exploiting the additional assumption made on the curl or the
		divergence of the target field, a localized upper bound on
		the quasi-interpolation error is derived. As an illustration,
		we show how to apply these results to the error analysis of the curl-curl
		problem associated with Maxwell's equations. 
	\end{abstract}
	\fi
	
	\maketitle
	
	\ifHAL
	\begin{abstract}
		We estimate best-approximation errors using vector-valued finite 
		elements for fields with low regularity in the scale of fractional-order
		Sobolev spaces. By assuming 
		that the target field enjoys an additional integrability property on its curl or 
		its divergence, we establish 
		upper bounds on these errors that can be localized to the mesh cells. 
		These bounds are derived using the quasi-interpolation errors with or
		without boundary prescription derived in [A. Ern and J.-L. Guermond,
		ESAIM Math. Model. Numer. Anal., 51 (2017), pp.~1367--1385].  
		In the present work, a localized upper bound on
		the quasi-interpolation error is derived
		by using the face-to-cell lifting operators analyzed in 
		[A. Ern and J.-L. Guermond, Found. Comput. Math., (2021)]
		and by exploiting the additional assumption made on the curl or the
		divergence of the target field. As an illustration,
		we show how to apply these results to the error analysis of the curl-curl
		problem associated with Maxwell's equations. 
	\end{abstract}
	\else
	\fi
	
	\section{Introduction}
	
	A central question in the finite element approximation theory is to establish 
	local upper bounds on the best-approximation error for functions
	that satisfy some minimal regularity assumptions 
	typically quantified in the scale of fractional-order Sobolev spaces.
	The goal of the present work is to derive some novel results in this context
	when the approximation is realized
	using Nédélec finite elements and Raviart--Thomas finite elements. 
	Most of our developments focus on the Nédélec finite elements 
	%for which the analysis is more elaborate
	since they require more elaborate arguments. The corresponding results for
	the Raviart--Thomas finite elements only improve marginally the state of the art from 
	the literature, and only a short discussion is provided in a specific
	section of the paper.
	
	We are interested in approximating fields with a smoothness Sobolev 
	index that is so low that one cannot invoke the canonical 
	interpolation operators associated 
	with the considered finite elements. In this case, the best-approximation error can
	be bounded by considering quasi-interpolation errors, such as those
	derived in \cite{ErnGu:17_quasi}.
	However, the resulting upper bound cannot be localized to the mesh cells
	if the regularity of the target function is only measured in the scale 
	of fractional-order Sobolev spaces. The lack of localization is relatively mild
	if no boundary conditions are prescribed in the finite
	element spaces, since in this case the quasi-interpolation error
	can still be bounded by local contributions involving the fractional-order
	Sobolev seminorm of the target function over patches of mesh cells instead of just each
	mesh cell individually. 
	The lack of localization is more significant if boundary conditions are
	additionally prescribed in the finite element spaces since in this case
	the upper bound on the quasi-interpolation error is global. This is not surprising
	as in this %second 
	situation the target function has not enough smoothness to
	have a well-defined trace at the boundary.
	The main contribution of this work is to show that the localization becomes 
	possible provided
	some (mild) additional assumptions are made on the integrability of the curl or
	the divergence of the target function. The main tool to achieve this result hinges on
	the face-to-cell lifting operators 
	introduced in \cite{ErnGu:21_FOCM}.
	The additional assumptions on the curl or the divergence allow us to give a 
	weak meaning to the trace of the target function
	in the dual space of a suitable fractional-order Sobolev space.
	
	In this work, the space dimension is $d=3$ for the Nédélec finite elements 
	and $d\ge 2$ for the Raviart--Thomas finite elements. 
	For $d=2$, the results for the Raviart--Thomas elements
	can be transposed to the Nédélec elements by invoking a rotation of angle $\frac{\pi}{2}$;
	details are omitted for brevity (see, \eg \cite[Sec.~15.3.1]{Ern_Guermond_FEs_I_2021}).
	We consider a polyhedral Lipschitz domain $\Dom\subset \Real^d$. 
	Moreover, we use boldface for $\Real^d$-valued 
	fields and linear spaces composed of such fields. For instance, for real
	numbers $r\ge 0$ and $p\in [1,\infty]$ (we assume $p\in [1,\infty)$
	if $r\not\in\polN$), $\bW^{r,p}(\Dom)$ denotes the (fractional-order) Sobolev space
	equipped with the Sobolev--Slobodeckij norm.  
	
	Let $\famTh$ denote a shape-regular family of affine, matching, 
	simplicial meshes 
	such that each mesh covers $\Dom$ exactly. 
	Let $\bP\upc_k(\calT_h)$ denote the $\Hrot$-conforming finite element space
	built on the mesh $\calT_h$ using the Nédélec finite element of degree $k\in \polN$
	(here, the superscript ${}\upc$ refers to the curl operator). 
	Given a target field $\bv \in \bW^{r,p}(\Dom)$, with $r>0$ possibly very small,
	our goal is to establish localized upper bounds
	on the best-approximation error
	\begin{equation} \label{eq:best_Pc}
		\inf_{\bv_h\in \bP\upc_k(\calT_h)} \|\bv-\bv_h\|_{\bL^p(\Dom)}.
	\end{equation}
	
	The first natural idea is to invoke the canonical interpolation operator
	for Nédélec finite elements, say $\inter_h\upc$. Since this operator can only 
	act on those fields having an
	integrable tangential trace along all the mesh edges, invoking the standard trace theory
	in Sobolev spaces (see, \eg \cite{Grisvard_1992}) shows that a suitable domain
	for the canonical interpolation operator $\inter_h\upc$ is $\bW^{r,p}(\Dom)$ 
	with $rp>d-1=2$ (and $r\ge 2$ if $p=1$). 
	Assume that the polynomial degree is such that
	$k\ge1$ if $p\in [1,2]$ and $k\ge0$ otherwise, and let
	$r\in (\frac{2}{p},k+1]$ if $p>1$ or $r\in [2,k+1]$ if $p=1$. Then, 
	it is well-known (see, \eg \cite{Monk_2003}
	or \cite[Sec.~16.2]{Ern_Guermond_FEs_I_2021}) that there is
	$c$ such that for all $\bv\in \bW^{r,p}(\Dom)$,
	all $K\in\calT_h$, and all $\hinH$, we have
	\begin{equation}
		\|\bv-\inter_h\upc(\bv)\|_{\bL^p(K)} \le c\, h_K^{r} |\bv|_{\bW^{r,p}(K)},
	\end{equation}
	where $h_K$ denotes the diameter of the mesh cell $K\in\calT_h$. 
	Here, the symbol $c$ denotes
	a generic positive constant whose value can change at each occurrence provided it
	only depends on the mesh shape-regularity, the space dimension, and the polynomial
	degree $k$ of the considered finite elements.
	Notice that $c$ is unbounded as $r\downarrow \frac{2}{p}$ if $p>1$.
	For instance, in the Hilbert setting where $p=2$, the minimal regularity requirement is 
	$\bv\in \bH^r(\Dom)$ with $r>1$, and $c$ is unbounded as $r\downarrow1$. 
	
	The requirement $r>\frac2p$ can be lowered
	to $r>\frac12$ (with $p=2$) by invoking more sophisticated results on traces
	derived in \cite{AmBDG:98} which, however, hinge on some additional integrability assumption 
	on $\ROT\bv$. To use these results, the edge-based
	degrees of freedom of the Nédélec finite element are extended by defining them
	using edge-to-cell lifting operators and an integration by parts formula
	(see, \eg \cite[Sec.~17.3]{Ern_Guermond_FEs_I_2021}). One can then show 
	(see \cite{BofGa:06}; see also \cite{AloVa:99,CiaZo:99,BeRoS:05} for slight variants) that
	for all $r\in (\frac12,1]$ and all $p>2$, there is $c$ such that 
	for all $\bv\in \bH^{r}(\Dom)$ with $\ROT\bv\in \bL^p(\Dom)$,
	all $K\in\calT_h$, and all $\hinH$, we have
	\begin{equation}
		\|\bv-\inter_h\upc(\bv)\|_{\bL^2(K)} \le c\, \Big(h_K^{r} |\bv|_{\bH^{r}(K)}
		+ h_K^{1+d(\frac12-\frac1p)}\|\ROT\bv\|_{\bL^p(K)}\Big),
	\end{equation}
	with $c$ unbounded as $r\downarrow \frac12$ or $p\downarrow2$.
	
	Unfortunately, the regularity assumption $r>\frac12$ is often not realistic in applications. 
	To go beyond this assumption, one can invoke the quasi-interpolation operators devised
	in \cite{ErnGu:17_quasi}. Recall that the construction of 
	these quasi-interpolation operators consists of first  
	projecting the target 
	field $\bv$ onto a fully discontinuous finite element space and then stitching 
	together the projected
	field to recover the desired conformity property by averaging the canonical degrees 
	of freedom of the projected field. Since the projected field is always piecewise smooth,
	this construction is always meaningful, regardless of the regularity of the target field 
	$\bv$. Let $\inter\upcav_h: \bL^1(\Dom)\to \bP\upc_k(\calT_h)$ denote the quasi-interpolation 
	operator thus constructed with the Nédélec finite elements. 
	Then, \cite[Thm.~22.6]{Ern_Guermond_FEs_I_2021} shows that there is 
	$c$ such that for all $r\in [0,k+1]$, all $p\in [1,\infty]$ if $r\in\polN$
	and $p\in[1,\infty)$ otherwise,
	all $\bv\in \bW^{r,p}(\Dom)$, all $K\in\calT_h$, and all $\hinH$, we have
	\begin{equation} \label{eq:inter_upcav}
		\|\bv-\inter_h\upcav(\bv)\|_{\bL^p(K)} \le c\, h_K^{r} |\bv|_{\bW^{r,p}(\Dom_K\upc)},
	\end{equation}
	where $\Dom_K\upc\eqq \interior\big(\bigcup_{K'\in\calT_K\upc} K'\big)$ and $\calT_K\upc$ 
	denotes the collection of mesh
	cells sharing at least one edge with $K$. We notice that a slight loss of localization occurs 
	in~\eqref{eq:inter_upcav} since the Sobolev--Slobodeckij seminorm on the right-hand side
	is evaluated over the macroelement $\Dom_K\upc$ and not just over $K$. 
	In the present work, we show that provided some (mild)
	additional integrability assumption is made on $\ROT\bv$, the 
	estimate~\eqref{eq:inter_upcav} can be localized to the mesh cells
	in $\calT\upc_K$; see Theorem~\ref{th:loc_av_Ned}
	and Corollary~\ref{cor:glob_av_Ned}.
	
	The loss of localization is more striking if one wants to additionally enforce a
	homogeneous boundary condition on the tangential component of the target field. We
	assume for
	simplicity that the condition is enforced over the whole boundary $\front$ of $\Dom$.
	Recall that the tangential trace operator $\gamma\upc:\Hrot \to \bH^{-\frac12}(\Dom)$
	is defined through a global integration by parts formula (see, \eg 
	\cite[Thm.~4.15]{Ern_Guermond_FEs_I_2021}) and that we have
	$\gamma\upc(\bv)\eqq \bv_{|\front}\times \bn_\Dom$ whenever the field $\bv$ is smooth enough, 
	where $\bn_\Dom$ denotes the unit outward normal to $\Dom$. Then, setting
	$\bP\upc_{k,0}(\calT_h)\eqq \{\bv_h\in\bP\upc_k(\calT_h)\tq \gamma\upc(\bv_h)
	=\bzero\}$, one is interested in establishing local upper bounds
	on the best-approximation error
	\begin{equation} \label{eq:best_Pcz}
		\inf_{\bv_h\in \bP\upc_{k,0}(\calT_h)} \|\bv-\bv_h\|_{\bL^p(\Dom)}.
	\end{equation}
	Let $\inter_{h0}\upcav:\bL^1(\Dom)\to \bP\upc_{k,0}(\calT_h)$ 
	denote the quasi-interpolation operator with homogeneous boundary prescription
	associated with the Nédélec finite elements. Then, 
	\cite[Thm.~22.14]{Ern_Guermond_FEs_I_2021}) shows that 
	for all $r\in [0,\frac1p)$, there is $c$ such that
	for all $\bv\in \bW^{r,p}(\Dom)$, and all $\hinH$, we have
	\begin{equation} \label{eq:inter_upcavz}
		\|\bv-\inter_{h0}\upcav(\bv)\|_{\bL^p(\Dom)} \le c\, h^{r}\ell_\Dom^{-r} \|\bv\|_{\bW^{r,p}(\Dom)},
	\end{equation}
	where $h\eqq \max_{K\in\calT_h} h_K$, $\ell_\Dom$ is a characteristic (global) length scale 
	associated with $\Dom$, and $\|\bv\|_{\bW^{r,p}(\Dom)} = 
	\|\bv\|_{\bL^p(\Dom)} + \ell_\Dom^r |\bv|_{\bW^{r,p}(\Dom)}$.
	Notice that the target field $\bv$ has not sufficient regularity to have a well-defined
	tangential trace on the boundary.
	The loss of localization in~\eqref{eq:inter_upcavz} arises when bounding 
	the quasi-interpolation error over those mesh cells that have at least one edge located on the boundary $\front$ 
	(the upper bound~\eqref{eq:inter_upcav} holds true for the other mesh cells).
	The presence of the global length scale $\ell_\Dom$ and of the full 
	Sobolev--Slobodeckij norm of $\bv$ instead of just the seminorm in \eqref{eq:inter_upcavz}
	comes from the need to invoke a Hardy inequality near the boundary (see the proof of 
	\cite[Thm.~6.4]{ErnGu:17_quasi}). In the present work, we show that provided some 
	(mild) additional integrability assumption is made on $\ROT\bv$, the 
	estimate~\eqref{eq:inter_upcavz} can be fully localized to the mesh cells; see again 
	Theorem~\ref{th:loc_av_Ned} and Corollary~\ref{cor:glob_av_Ned}.

	\section{Main results on Nédélec finite elements} \label{sec:results}
	
	In this section, we first state our main results and then present their proofs.
	
	\subsection{Statement of the main results}
	
	Let us first observe that the domain of the tangential trace operator can be
	extended to $\bY\upc(\Dom)\eqq \{\bv\in \Ldeuxd\st \ROT\bv\in \bL^q(\Dom)\}$
	for all $q\in (\frac{2d}{2+d},2]$. Indeed, for all $\bv\in \bY\upc(\Dom)$,
	$\gamma\upc(\bv)\in \bH^{-\frac12}(\front)$ can still be defined by duality 
	by setting for all $\bw\in \bH^{\frac12}(\front)$,
	\begin{equation} \label{eq:ext_tracec_glob}
		\langle \gamma\upc(\bv),\bw\rangle_{\front} \eqq \int_\Dom \Big(\bv\SCAL\ROT\bl(\bw)-
		(\ROT\bv)\SCAL\bl(\bw)\Big)\dif x,
	\end{equation}
	where $\bl(\bw)$ denotes a lifting of $\bw$ in
	$\bH^1(\Dom)$. Indeed, owing to the Sobolev embedding theorem,
	and the fact that $q> \frac{2d}{2+d}$, we infer that
	$\bH^1(\Dom)\hookrightarrow \bL^{q'}(\Dom)$ with $\frac1q+\frac{1}{q'}=1$,
	so that the second term on the right-hand side of \eqref{eq:ext_tracec_glob}
	is meaningful owing to H\"older's inequality.
	We can now state our main result. For simplicity, we estimate the
	quasi-interpolation error only in the $\bL^2$-norm.
	
	\begin{theorem}[Localized quasi-interpolation error estimate for Nédélec elements]
		\label{th:loc_av_Ned}
		For all $r\in (0,1]$ and all $q\in (\frac{2d}{2+d},2]$,
		there is $c$ such that for
		all $\bv\in \bH^r(\Dom)$ with $\ROT\bv\in \bL^q(\Dom)$, all $K\in\calT_h$,
		and all $\hinH$, we have
		\begin{equation} \label{eq:loc_avc}
			\|\bv-\inter\upcav_h(\bv)\|_{\bL^2(K)} \le
			c\sum_{K'\in\calT_K\upc} \Big\{ 
			h_{K'}^r |\bv|_{\bH^r(K')} + h_{K'}^{1+d(\frac12-\frac1q)} \|\ROT\bv\|_{\bL^q(K')}\Big\}.
		\end{equation}
		Moreover, assuming that $\gamma\upc(\bv)=\bzero$, we also have
		\begin{equation} \label{eq:loc_avcz}
			\|\bv-\inter\upcav_{h0}(\bv)\|_{\bL^2(K)} \le
			c\sum_{K'\in\calT_K\upc} \Big\{ 
			h_{K'}^r |\bv|_{\bH^r(K')} + h_{K'}^{1+d(\frac12-\frac1q)} \|\ROT\bv\|_{\bL^q(K')}\Big\}.
		\end{equation}
	\end{theorem}
	
	Squaring the above inequalities, summing over the mesh cells, and observing that
	the cardinality of the set $\{K'\in\calT_h\tq K\in\calT\upc_{K'}\}$ is uniformly
	bounded for all $K\in\calT_h$ and all $\hinH$, we infer the following result.
	
	\begin{corollary}[Localized best-approximation error for Nédélec elements]
		\label{cor:glob_av_Ned}
		For all $r\in (0,1]$ and all $q\in (\frac{2d}{2+d},2]$,
		there is $c$ such that for
		all $\bv\in \bH^r(\Dom)$ with $\ROT\bv\in \bL^q(\Dom)$, 
		and all $\hinH$, we have
		\begin{equation} \label{eq:glob_avc}
			\inf_{\bv_h\in \bP\upc_k(\calT_h)}\|\bv-\bv_h\|_{\bL^2(\Dom)} \le
			c\, \bigg\{ \sum_{K\in\calT_h} \Big\{ 
			h_{K}^{2r} |\bv|_{\bH^r(K)}^2 + h_{K}^{2+2d(\frac12-\frac1q)} \|\ROT\bv\|_{\bL^q(K)}^2\Big\} \bigg\}^{\frac12}.
		\end{equation}
		Moreover, assuming that $\gamma\upc(\bv)=\bzero$, we also have
		\begin{equation} \label{eq:glob_avcz}
			\inf_{\bv_h\in \bP_{k,0}\upc(\calT_h)}\|\bv-\bv_h\|_{\bL^2(\Dom)} \le
			c\,\bigg\{\sum_{K\in\calT_h} \Big\{
			h_{K}^{2r} |\bv|_{\bH^r(K)}^2 + h_{K}^{2+2d(\frac12-\frac1q)} \|\ROT\bv\|_{\bL^q(K)}^2\Big\} \bigg\}^{\frac12}.
		\end{equation}
	\end{corollary}
	
	\begin{remark}[Exponent]
		Observe that $2+2d(\frac12-\frac1q)=2d(\frac{d+2}{2d}-\frac1q)>0$
		since $q\in (\frac{2d}{2+d},2]$. Moreover, we have
		$2+2d(\frac12-\frac1q)=2$ for $q=2$.
	\end{remark}

	\subsection{Preliminary: localizing the tangential trace to the mesh faces}
	
	This section collects some preliminary results needed in the proof of 
	Theorem~\ref{th:loc_av_Ned}. These results are drawn from \cite{ErnGu:21_FOCM}
	and are briefly restated here for the reader's convenience.
	
	Let $K\in \calT_h$ be a mesh cell, let $\calF_K$ be the collection of the
	faces of $K$, and let $F\in \calF_K$.
	To define a tangential trace that is localized to the mesh face $F$, we introduce 
	the local functional space $\bV\upc(K)\eqq \{\bv\in \bL^p(K)\st \ROT\bv\in \bL^q(K)\}$
	with $q\in (\frac{2d}{2+d},2]$ (as above) and $p>2$. We equip this space with the 
	(dimensionally consistent) norm
	\begin{equation} \label{eq:norm_VupcK}
		\|\bv\|_{\bV\upc(K)} := \|\bv\|_{\bL^p(K)} + h_K^{1+d(\frac1p-\frac1q)} \|\ROT\bv\|_{\bL^q(K)}.
	\end{equation}
	Let $\varrho\in (2,p]$ be such that
	$q\ge \frac{\varrho d}{\varrho + d}$ (this is indeed possible since the function
	$x\mapsto \frac{xd}{x+d}$ is increasing on $[2,\infty)$). Let $\varrho'\in [1,2)$
	be such that $\frac{1}{\varrho}+\frac{1}{\varrho'}=1$. We consider the 
	(fractional-order) Sobolev space $\bW^{\frac{1}{\varrho},\varrho'}(F)$,
	equipped with the (dimensionally consistent) norm 
	\begin{equation} \label{eq:def_norm_varrho}
		\|\bphi\|_{\bW^{\frac{1}{\varrho},\varrho'}(F)} 
		\eqq \|\bphi\|_{\bL^{\varrho'}(F)}+h_F^{\frac{1}{\varrho}}
		|\bphi|_{\bW^{\frac{1}{\varrho},\varrho'}(F)}.
	\end{equation}
	Let $(\bW^{\frac{1}{\varrho},\varrho'}(F))'$ denote the dual space of 
	$\bW^{\frac{1}{\varrho},\varrho'}(F)$.
	It is shown in \cite[Equ.~(5.5)]{ErnGu:21_FOCM} that upon introducing suitable face-to-cell lifting
	operators, it is possible to define  a tangential
	trace operator localized to the mesh face $F$ through an integration by parts formula, namely $\gamma\upc_{K,F}:\bV\upc(K) \to
	(\bW^{\frac{1}{\varrho},\varrho'}(F))'$, such that the following two properties hold:
	\textup{(i)} $\gamma_{K,F}(\bv) = (\bv\times\bn_K)_{|F}$ whenever the field $\bv$ is smooth,
	where $\bn_K$ denotes the unit normal to $\partial K$ pointing outward $K$;
	\textup{(ii)} There is $c$ such that for all $\bv\in \bV\upc(K)$, all $K\in\calT_h$,
	and all $\hinH$,
	\begin{equation} \label{eq:bnd_gamma}
		\|\gamma_{K,F}\upc(\bv)\|_{(\bW^{\frac{1}{\varrho},\varrho'}(F))'} \le c\, 
		h_K^{-\frac{1}{\varrho}+d(\frac{1}{\varrho}-\frac1p)} \|\bv\|_{\bV\upc(K)}.
	\end{equation}
	
	Let $\bpolN^k(K)$ be composed of the
	restriction to $K$ of the Nédélec polynomials of order $k\in\polN$. 
	Let us set $\bpolN^k(F)\eqq \gamma_{K,F}\upc(\bpolN^k(K))$ for all $F\in\calF_K$.
	In this work, we need to invoke the following inverse inequality. 
	
	\begin{lemma}[Inverse inequality on $F$] \label{eq:inv_ineq_F}
		Let $t\in [1,\infty]$. There is $c$
		such that for all $\bphi_h\in \bpolN^k(F)$, all $K\in\calT_h$, all $F\in\calF_K$,
		and all $\hinH$,
		\begin{equation}
			\|\bphi_h\|_{\bL^t(F)} \le c\, h_K^{(d-1)(\frac{1}{t}-\frac{1}{\varrho})}
			\|\bphi_h\|_{(\bW^{\frac{1}{\varrho},\varrho'}(F))'}.
		\end{equation}
	\end{lemma}
	
	\begin{proof}
		Let $\bxi_h\in \bpolN^k(F)$. Recalling the definition~\eqref{eq:def_norm_varrho}
		of the $\|\SCAL\|_{\bW^{\frac{1}{\varrho},\varrho'}(F)}$-norm and invoking an inverse inequality
		on $\bpolN^k(F)$ (see, \eg \cite[Sec.~12.1]{Ern_Guermond_FEs_I_2021}, 
		and observe that $\bpolN^k(F) \subset \bpolP^{k+1}_{d-1}\circ \bT_F^{-1}$, where
		$\bT_F:\wS^{d-1}\to F$ is the geometric mapping from the reference $(d-1)$-dimensional
		simplex $\wS^{d-1}$ to $F$ and where $\bpolP^{k+1}_{d-1}$ is composed
		of $\Real^d$-valued, $(d-1)$-variate polynomials of order at most $(k+1)$), we infer that
		\[
		\|\bxi_h\|_{\bW^{\frac{1}{\varrho},\varrho'}(F)} \le c\, h_K^{(d-1)(\frac{1}{\varrho'}-\frac{1}{t'})}
		\|\bxi_h\|_{\bL^{t'}(F)},
		\]
		with $t'\in [1,\infty]$ such that $\frac{1}{t}+\frac{1}{t'}=1$.
		This implies that
		\[
		\|\bphi_h\|_{\bL^t(F)}  = \sup_{\bxi_h\in \bpolN^k(F)} \frac{\int_F \bphi_h\SCAL \bxi_h\dif s}{\|\bxi_h\|_{\bL^{t'}(F)}} \le ch_K^{(d-1)(\frac{1}{\varrho'}-\frac{1}{t'})} \sup_{\bxi_h\in \bpolN^k(F)} \frac{\int_F \bphi_h\SCAL \bxi_h\dif s}{\|\bxi_h\|_{\bW^{\frac{1}{\varrho},\varrho'}(F)}}.
		\]
		Since $\int_F \bphi_h\SCAL \bxi_h\dif s = \langle \bphi_h,\bxi_h\rangle_{(\bW^{\frac{1}{\varrho},\varrho'}(F))',\bW^{\frac{1}{\varrho},\varrho'}(F)}$, the assertion follows from the definition of the dual norm in $(\bW^{\frac{1}{\varrho},\varrho'}(F))'$ and the identity $\frac{1}{\varrho'}-\frac{1}{t'}=\frac{1}{t}-\frac{1}{\varrho}$.
	\end{proof}
	
	Let us set
	\begin{align}
		\bV\upc(\Dom)&\eqq \{\bv\in \bL^p(\Dom)\st \ROT\bv\in \bL^q(\Dom)\}, \label{eq:def_Vc_Dom} \\
		\bV\upc_0(\Dom)&\eqq \{\bv\in \bV\upc(\Dom)\tq \gamma\upc(\bv)=\bzero\}.
	\end{align}
	(Observe that the tangential trace operator $\gamma\upc$ is meaningful
	on $\bV\upc(\Dom)$ since $p>2$.) \ale{Proceeding as in 
		\cite[Thm.~4.15]{Ern_Guermond_FEs_I_2021}, one can show that $\bV\upc_0(\Dom)$
		coincides with the closure of $\bC_0^\infty(\Dom)$ in $\bV\upc(\Dom)$.}
	We notice that for all $\bv\in \bV\upc(\Dom)$
	and all $K\in\calT_h$, we have $\bv_{|K}\in \bV\upc(K)$. We also define the broken
	version of $\bV\upc(\Dom)$ as follows:
	\begin{equation}
		\bV\upc(\calT_h)\eqq \{\bv\in \bL^p(\Dom)\st \bv_{|K}\in \bV\upc(K),\, \forall K\in\calT_h\}.
	\end{equation}
	
	The collection of the mesh faces, $\calFh$, is split into the collection of the
	mesh interfaces, $\calFhi$, and the collection of the mesh boundary faces, $\calFhb$.
	For all $F\in\calFhi$, there are two distinct mesh cells $K_-,K_+\in\calT_h$ such that
	$F=\partial K_-\cap \partial K_+$. For all $F\in \calFhb$, there is one mesh cell
	$K_-\in\calT_h$ such that $F=\partial K_-\cap \front$. 
	For every field $\bv\in \bV\upc(\calT_h)$, the jump of the tangential component across
	the mesh interface $F=\partial K_-\cap \partial K_+\in\calFhi$ is defined as
	\begin{equation}
		\jump{\gamma_{K,F}(\bv)}_F \eqq \gamma_{K_-,F}(\bv_{|K_-})+\gamma_{K_+,F}(\bv_{|K_+}).
	\end{equation}
	Moreover, for every mesh boundary face $F=\partial K_-\cap \front\in\calFhb$,
	we conventionally set
	\begin{equation}
		\jump{\gamma_{K,F}(\bv)}_F \eqq \gamma_{K_-,F}(\bv_{|K_-}).
	\end{equation}
	
	\begin{lemma}[Vanishing jumps and boundary traces]
		\label{lem:jump_trace}
		\textup{(i)} For all $\bv\in\bV\upc(\Dom)$, we have $\jump{\gamma_{K,F}(\bv)}_F=\bzero$ for
		all $F\in\calFhi$.
		\textup{(ii)} For all $\bv\in\bV\upc_0(\Dom)$, we additionally have 
		$\jump{\gamma_{K,F}(\bv)}_F=\bzero$ for all $F\in\calFhb$.
	\end{lemma}
	
	\begin{proof}
		Let $\bv\in\bV\upc(\Dom)$ and let $F\in\calFh$. For all $K\in\calT_h$
		such that $F\in\calFK$, $\gamma_{K,F}(\bv)$ is defined in
		\cite[Equ.~(5.5)]{ErnGu:21_FOCM} so that
		\[
		\langle \gamma_{K,F}(\bv),\bphi\rangle_{(\bW^{\frac{1}{\varrho},\varrho'}(F))',
			\bW^{\frac{1}{\varrho},\varrho'}(F)} \eqq \int_K \Big( 
		\bv\SCAL\ROT E_F^K(\bphi) - (\ROT\bv)\SCAL E_F^K(\bphi)\Big)\dif x,
		\]
		for all $\bphi\in \bW^{\frac{1}{\varrho},\varrho'}(F)$,
		where $E_F^K:\bW^{\frac{1}{\varrho},\varrho'}(F)\to \bW^{1,\varrho'}(K)$ is the
		face-to-cell lifting operator from \cite[Def.~5.1]{ErnGu:21_FOCM}.
		\\
		\textup{(i)} Let $F=\partial K_-\cap\partial K_+\in\calFhi$.
		We define the global lifting operator $E_F^\Dom :
		\bW^{\frac{1}{\varrho},\varrho'}(F)\to \bW^{1,\varrho'}(\Dom)$ such that
		$E_F^\Dom(\bphi)_{|K_\pm} \eqq E_F^{K_\pm}(\bphi)$ and $E_F^\Dom(\bphi)=\bzero$
		otherwise. Summing the above identity for $K\in\{K_-,K_+\}$, we infer that
		\begin{equation} \label{id_jump}
			\langle \jump{\gamma_{K,F}(\bv)}_F,\bphi\rangle_{(\bW^{\frac{1}{\varrho},\varrho'}(F))',
				\bW^{\frac{1}{\varrho},\varrho'}(F)} \eqq \int_\Dom \Big( 
			\bv\SCAL\ROT E_F^\Dom(\bphi) - (\ROT\bv)\SCAL E_F^\Dom(\bphi)\Big)\dif x,
		\end{equation}
		for all $\bphi\in \bW^{\frac{1}{\varrho},\varrho'}(F)$.  Since,
		by construction, $E_F^\Dom(\bphi)$ has a zero trace at the
		boundary of $\Dom$ \ale{and since $\bV\upc_0(\Dom)$
			coincides with the closure of $\bC_0^\infty(\Dom)$ in $\bV\upc(\Dom)$}, we conclude that
		\[
		\langle \jump{\gamma_{K,F}(\bv)}_F,\bphi\rangle_{(\bW^{\frac{1}{\varrho},\varrho'}(F))',\bW^{\frac{1}{\varrho},\varrho'}(F)}=0\ale{.}
		\] 
		\ale{S}ince $\bphi$ is arbitrary in 
		$\bW^{\frac{1}{\varrho},\varrho'}(F)$, this implies that $\jump{\gamma_{K,F}(\bv)}_F=\bzero$.
		\\
		\textup{(ii)} If $F=\partial K_-\cap\front\in\calFhb$, we set
		$E_F^\Dom(\bphi)_{|K_-} \eqq E_F^{K_-}(\bphi)$ and $E_F^\Dom(\bphi)=\bzero$
		otherwise. This yields again \eqref{id_jump}, and invoking that $\gamma\upc(\bv)=\bzero$
		still gives 
		$\langle \jump{\gamma_{K,F}(\bv)}_F,\bphi\rangle_{(\bW^{\frac{1}{\varrho},\varrho'}(F))',
			\bW^{\frac{1}{\varrho},\varrho'}(F)}=0$, whence the assertion.
	\end{proof}
	
	\subsection{Proof of~\eqref{eq:loc_avc}}
	
	Let us start with a preliminary result of independent interest.
	For all $K\in\calT_h$, let $\Pi^0_K:\bL^1(K)\to \bpolP^0(K)$ denote the (local)
	$\bL^2$-orthogonal projection onto $\bpolP^0(K)$ (that is, $\Pi_K^0(\bv)$ is the
	mean value of $\bv$ over $K$).
	
	\begin{lemma}[Localized quasi-interpolation error in $\bV\upc(\Dom)$]
		\label{lem:quasi_int_Vc}
		For all $p>2$ and all $q\in (\frac{2d}{2+d},2]$, there is $c$ such that
		for all $\bv\in \bV\upc(\Dom)$, all $K\in\calT_h$, and all $\hinH$, we have
		\begin{equation} \label{eq:quasi_int_Vc}
			\|\bv-\inter\upcav_h(\bv)\|_{\bL^2(K)} \le
			c\, h_K^{d(\frac12-\frac1p)} \sum_{K'\in\calT_K\upc} \Big( 
			\|\bv-\Pi^0_{K'}(\bv)\|_{\bL^p(K')} + h_{K'}^{1+d(\frac1p-\frac1q)} \|\ROT\bv\|_{\bL^q(K')}\Big).
		\end{equation}
	\end{lemma}
	
	\begin{proof}
		Let $\Pi\upc_h:\bL^1(\Dom)\to \bP\upcb_k(\calT_h)\eqq \{\bv_h\in \bL^1(\Dom)\tq
		\bv_{h|K}\in \bpolN^k(K),\, \forall K\in\calT_h\}$ denote the global
		$\bL^2$-orthogonal projection onto the broken Nédélec finite element space 
		$\bP\upcb_k(\calT_h)$.
		Let $\Pi\upc_K:\bL^1(K)\to \bpolN^k(K)$ denote the local
		$\bL^2$-orthogonal projection onto $\bpolN^k(K)$, so that we have 
		$\Pi\upc_h(\bv)_{|K}=\Pi\upc_K(\bv_{|K})$ for all $\bv\in \bL^1(\Dom)$
		and all $K\in\calT_h$. Recall
		(see \cite[Sec.~5]{ErnGu:17_quasi}) that we have $\inter\upcav_h \eqq
		\calJ_h\upcav \circ \Pi\upc_h$, where the %averaging 
		operator 
		$\calJ_h\upcav:\bP\upcb_k(\calT_h)\to \bP\upc_k(\calT_h)$ is built by averaging
		the canonical degrees of freedom, see \cite[Sec.~4.2]{ErnGu:17_quasi}. 
		The triangle inequality implies that 
		\[
		\|\bv-\inter\upcav_h(\bv)\|_{\bL^2(K)} \le \|\bv-\Pi\upc_K(\bv)\|_{\bL^2(K)}
		+ \|(I-\calJ\upcav_h)(\Pi_h\upc(\bv))\|_{\bL^2(K)} \qqe \term_1+\term_2.
		\]
		Since $\bpolP^0(K)\subset \bpolN^k(K)$, standard properties of the 
		$\bL^2$-orthogonal projection and H\"older's inequality (recall that $p>2$) imply that
		\[
		\term_1 \le \|\bv-\Pi^0_K(\bv)\|_{\bL^2(K)}
		\le c\, h_K^{d(\frac12-\frac1p)}\|\bv-\Pi^0_K(\bv)\|_{\bL^p(K)}.
		\]
		Moreover, \cite[Lem.~4.3]{ErnGu:17_quasi} followed by Lemma~\ref{eq:inv_ineq_F} (with $t=2$) give (recall that the value of $c$ can change at each occurrence)
		\[
		\term_2 \le c\, h_K^{\frac12} \sum_{F\in\check\calF_K^\circ} \|
		\jump{\gamma_{K,F}\upc(\Pi\upc_h(\bv))}_F \|_{\bL^2(F)}
		\le c\, h_K^{\frac12+(d-1)(\frac12-\frac{1}{\varrho})} \sum_{F\in\check\calF_K^\circ} \|
		\jump{\gamma_{K,F}\upc(\Pi\upc_h(\bv))}_F \|_{(\bW^{\frac{1}{\varrho},\varrho'}(F))'},
		\]
		where $\check\calF_K^\circ$ denotes the collection of the mesh interfaces sharing at least
		an edge with $K$. Observing that $\gamma_{K,F}\upc(\bv)=\bzero$ since $\bv\in \bV\upc(\Dom)$
		(see Lemma~\ref{lem:jump_trace}(i)), we infer that
		\[
		\term_2 \le c\, 
		h_K^{\frac12+(d-1)(\frac12-\frac{1}{\varrho})} \sum_{F\in\check\calF_K^\circ} \|
		\jump{\gamma_{K,F}\upc(\bv-\Pi\upc_h(\bv))}_F \|_{(\bW^{\frac{1}{\varrho},\varrho'}(F))'}.
		\]
		By definition of the jump operator, invoking the triangle inequality and recalling the definition of the set $\calT_K\upc$ yields
		\[
		\term_2 \le c\, 
		h_K^{\frac12+(d-1)(\frac12-\frac{1}{\varrho})} \sum_{K'\in \calT_K\upc} \sum_{F\in\calF_{K'}\cap\calFhi} \|
		\gamma_{K',F}\upc(\bv-\Pi\upc_{K'}(\bv)) \|_{(\bW^{\frac{1}{\varrho},\varrho'}(F))'}.
		\]
		Owing to the bound~\eqref{eq:bnd_gamma} and the shape-regularity of the mesh sequence, we infer that
		\begin{equation} \label{eq:bound_on_t2}
			\term_2 \le c\, h_K^{d(\frac12-\frac1p)} 
			\sum_{K'\in \calT_K\upc} \|\bv-\Pi\upc_{K'}(\bv)\|_{\bV\upc(K')}.
		\end{equation}
		Invoking the triangle inequality and recalling the definition~\eqref{eq:norm_VupcK} of
		the norm equipping $\bV\upc(K')$ gives for all $K'\in \calT_K\upc$,
		\[
		\|\bv-\Pi\upc_{K'}(\bv)\|_{\bV\upc(K')} \le
		\|\bv-\Pi^0_K(\bv)\|_{\bL^p(K')} + h_{K'}^{1+d(\frac1p-\frac1q)} \|\ROT\bv\|_{\bL^q(K')}
		+ \|(\Pi\upc_{K'}-\Pi^0_{K'})(\bv)\|_{\bV\upc(K')},
		\]
		where we used that $\Pi^0_K(\bv)$ is a constant field in $K'$. Invoking 
		inverse inequalities on the first line, the triangle inequality and standard properties of
		the $\bL^2$-orthogonal projection on the second line, and, finally, 
		H\"older's inequality  (recall that $p>2$) 
		and the regularity of the mesh sequence on the third line, we also have
		\begin{align*}
			\|(\Pi\upc_{K'}-\Pi^0_{K'})(\bv)\|_{\bV\upc(K')}
			&\le c\, h_{K'}^{d(\frac1p-\frac12)} \|(\Pi\upc_{K'}-\Pi^0_{K'})(\bv)\|_{\bL^2(K')} \\
			&\le c\, h_{K'}^{d(\frac1p-\frac12)} \|(I-\Pi^0_{K'})(\bv)\|_{\bL^2(K')} \\
			&\le c\, \|(I-\Pi^0_{K'})(\bv)\|_{\bL^p(K')}.
		\end{align*}
		Combining this bound with~\eqref{eq:bound_on_t2} and recalling the bound on $\term_1$
		completes the proof.
	\end{proof}
	
	\medskip
	
	\begin{proof}[Proof of~\eqref{eq:loc_avc}.]
		Let $r\in (0,1]$ and $q\in (\frac{2d}{2+d},2]$. Consider a field
		$\bv\in \bH^r(\Dom)$ with $\ROT\bv\in \bL^q(\Dom)$. The Sobolev 
		embedding theorem implies that $\bv \in \bV\upc(\Dom)$ (indeed,
		we can take $p\eqq \frac{2d}{d-2r}>2$ if $r<\frac{d}{2}$ and any $p\in (2,\infty)$
		otherwise). Applying Lemma~\ref{lem:quasi_int_Vc}, re-arranging the factors involving
		the local mesh sizes on the right-hand side of~\eqref{eq:quasi_int_Vc} and using
		the shape-regularity of the mesh sequence, we infer that
		\[
		\|\bv-\inter\upcav_h(\bv)\|_{\bL^2(K)} \le
		c\,  \sum_{K'\in\calT_K\upc} \Big( h_{K'}^{d(\frac12-\frac1p)}
		\|\bv-\Pi^0_{K'}(\bv)\|_{\bL^p(K')} + h_{K'}^{1+d(\frac12-\frac1q)} \|\ROT\bv\|_{\bL^q(K')}\Big).
		\]
		Owing to \cite[Equ.~(17.19)]{Ern_Guermond_FEs_I_2021} and since $\Pi^0_{K'}(\bv)$ 
		is a constant field in $K'$, we infer that
		\[
		h_{K'}^{d(\frac12-\frac1p)} \|\bv-\Pi^0_{K'}(\bv)\|_{\bL^p(K')}
		\le c\, \big( \|\bv-\Pi^0_{K'}(\bv)\|_{\bL^2(K')} + h_{K'}^r|\bv|_{\bH^r(K')}\big),
		\]
		so that invoking the (fractional) Poincar\'e--Steklov inequality in $K'$ 
		(see, \eg  \cite[Sec.~12.3.1]{Ern_Guermond_FEs_I_2021}) gives
		\[
		h_{K'}^{d(\frac12-\frac1p)} \|\bv-\Pi^0_{K'}(\bv)\|_{\bL^p(K')}
		\le c\, h_{K'}^r|\bv|_{\bH^r(K')}.
		\]
		Combining the above bounds completes the proof.
	\end{proof}
	
	\subsection{Proof of~\eqref{eq:loc_avcz}}
	
	The proof of~\eqref{eq:loc_avcz} follows from minor adaptations of the arguments
	to prove~\eqref{eq:loc_avc}. The most relevant change hinges on the following result.
	
	\begin{lemma}[Localized quasi-interpolation error in $\bV\upc_0(\Dom)$]
		\label{lem:quasi_int_Vcz}
		For all $p>2$ and all $q\in (\frac{2d}{2+d},2]$, there is $c$ such that
		for all $\bv\in \bV\upc_0(\Dom)$, all $K\in\calT_h$, and all $\hinH$, we have
		\begin{equation} \label{eq:quasi_int_Vcz}
			\|\bv-\inter\upcav_{h0}(\bv)\|_{\bL^2(K)} \le
			c\, h_K^{d(\frac12-\frac1p)} \sum_{K'\in\calT_K\upc} \Big( 
			\|\bv-\Pi^0_{K'}(\bv)\|_{\bL^p(K')} + h_{K'}^{1+d(\frac1p-\frac1q)} \|\ROT\bv\|_{\bL^q(K')}\Big).
		\end{equation}
	\end{lemma}
	
	\begin{proof}
		The only relevant difference with the proof of Lemma~\ref{lem:quasi_int_Vc}
		is in the upper bound estimate for $\term_2$. 
		Owing to the homogeneous boundary prescription,
		we now have $\inter_{h0}\upcav\eqq
		\calJ_{h0}\upcav \circ \Pi\upc_h$, where the averaging operator 
		$\calJ_{h0}\upcav:\bP\upcb_k(\calT_h)\to \bP\upc_{k,0}(\calT_h)$ is built as 
		$\calJ_h\upcav$ with the additional requirement that all the degrees of freedom 
		attached to the mesh edges and faces located on the boundary $\front$ are zero
		(see \cite[Sec.~6.2]{ErnGu:17_quasi}). As a result, we now have
		\[
		\term_2 \le c\, h_K^{\frac12} \sum_{F\in\check\calF_K} \|
		\jump{\gamma_{K,F}\upc(\Pi\upc_h(\bv))}_F \|_{\bL^2(F)},
		\]
		where $\check\calF_K$ denotes the collection of the mesh faces (and not only the mesh 
		interfaces) sharing at least an edge with $K$. The bound on $\term_2$ is obtained by using the same
		arguments as above and the fact that $\jump{\gamma_{K,F}(\bv)}_F=\bzero$
		for every mesh boundary face $F\in\calFhb$ since $\bv\in \bV\upc_0(\Dom)$
		(see Lemma~\ref{lem:jump_trace}(ii)).
	\end{proof}

	\section{Raviart--Thomas finite elements} \label{Sec:review_Raviart_Thomas}
	
	The discussion for the Raviart--Thomas finite elements goes along the same lines
	as for the Nédélec finite elements, except that the minimal 
	regularity requirements on the target
	field are less demanding when using the canonical interpolation operator
	(or any extension thereof) since one only needs to give a meaning to the trace of
	the normal component of the target field on the mesh faces. 
	
	Let $\bP\upd_k(\calT_h)$ denote the $\Hdiv$-conforming finite element space
	built on the mesh $\calT_h$ using the Raviart--Thomas finite element of degree $k\in\polN$ 
	(here, the superscript ${}\upd$ refers to the divergence operator) and consider the 
	best-approximation error 
	\begin{equation} \label{eq:best_Pd}
		\inf_{\bv_h\in \bP\upd_k(\calT_h)} \|\bv-\bv_h\|_{\bL^p(\Dom)},
	\end{equation}
	for an arbitrary target field $\bv \in \bW^{r,p}(\Dom)$, with $r>0$ possibly very small.
	Let $\inter_h\upd$ be the canonical interpolation operator associated with the
	Raviart--Thomas finite elements. It is well-known (see, \eg \cite{BoBrF:13}
	or \cite[Sec.~16.1]{Ern_Guermond_FEs_I_2021}) that for all $r\in (\frac{1}{p},k+1]$ if $p>1$
	and $r\in [1,k+1]$ if $p=1$, there is $c$
	such that for all $\bv\in \bW^{r,p}(\Dom)$,
	all $K\in\calT_h$, and all $\hinH$, we have
	\begin{equation}
		\|\bv-\inter_h\upd(\bv)\|_{\bL^p(K)} \le c\, h_K^{r} |\bv|_{\bW^{r,p}(K)},
	\end{equation}
	with $c$ unbounded as $r\downarrow \frac{1}{p}$ if $p>1$.
	For instance, in the Hilbert setting where $p=2$, the minimal regularity requirement is 
	$\bv\in \bH^r(\Dom)$ with $r>\frac12$, and $c$ is unbounded as $r\downarrow\frac12$.
	Moreover, by invoking more sophisticated results on traces and assuming some
	additional integrability property on $\DIV\bv$, one can show (see, \eg
	\cite[Sec.~17.2]{Ern_Guermond_FEs_I_2021}) that 
	for all $r\in (0,1]$ and all $q>\frac{2d}{2+d}$, there is $c$ such that
	for all $\bv\in \bH^r(\Dom)$
	with $\DIV\bv\in L^q(\Dom)$, all $K\in\calT_h$, and all $\hinH$, we have
	\begin{equation} \label{loc_Ihd}
		\|\bv-\inter_h\upd(\bv)\|_{\bL^2(K)} \le c\, \Big(h_K^{r} |\bv|_{\bH^{r}(K)}
		+ h_K^{1+d(\frac12-\frac1q)}\|\DIV\bv\|_{L^q(K)}\Big),
	\end{equation}
	with $c$ unbounded as $r\downarrow0$ or $q\downarrow\frac{2d}{2+d}$.
	
	It is also possible to consider the quasi-interpolation operators devised 
	in \cite{ErnGu:17_quasi}. Let $\inter\updav_h: \bL^1(\Dom)\to \bP\upd_k(\calT_h)$ 
	denote the quasi-interpolation 
	operator associated with the Raviart--Thomas finite element. 
	Then, \cite[Thm.~22.6]{Ern_Guermond_FEs_I_2021} shows that there is 
	$c$ such that for all $r\in [0,k+1]$, all $p\in [1,\infty]$ if $r\in\polN$
	and $p\in[1,\infty)$ otherwise,
	all $\bv\in \bW^{r,p}(\Dom)$, all $K\in\calT_h$, and all $\hinH$, we have
	\begin{equation} \label{eq:inter_updav}
		\|\bv-\inter_h\updav(\bv)\|_{\bL^p(K)} \le c\, h_K^{r} |\bv|_{\bW^{r,p}(\Dom_K\upd)},
	\end{equation}
	$\Dom_K\upd\eqq \interior\big(\bigcup_{K'\in\calT_K\upd} K'\big)$ and $\calT_K\upd$ 
	denotes the collection of mesh
	cells sharing at least one face with $K$. We notice again 
	that a slight loss of localization occurs 
	in~\eqref{eq:inter_updav} since the Sobolev--Slobodeckij seminorm on the right-hand side
	is evaluated over the macroelement $\Dom_K\upd$ and not just over $K$. 
	% In the present work, we show that provided some (mild)
	% additional integrability assumption is made on $\DIV\bv$, the 
	% estimate~\eqref{eq:inter_updav} can be localized to the mesh
	% cells in $\calT\upd_K$; see Theorem~\ref{th:loc_av_RT}
	% and Corollary~\ref{cor:glob_av_RT}. 
	One can also enforce a homogeneous condition on the normal component
	of the target field over the whole boundary $\front$ of $\Dom$ (for simplicity).
	Recall that the normal trace operator $\gamma\upd:\Hdiv \to \bH^{-\frac12}(\Dom)$
	is defined through a global integration by parts formula (see, \eg 
	\cite[Thm.~4.15]{Ern_Guermond_FEs_I_2021}) and that we have
	$\gamma\upd(\bv)\eqq \bv_{|\front}\SCAL \bn_\Dom$ whenever the field $\bv$ is smooth enough. 
	Then, setting
	$\bP\upd_{k,0}(\calT_h)\eqq \{\bv_h\in\bP\upd_k(\calT_h)\tq \gamma\upd(\bv_h)
	=0\}$, one is interested in establishing local upper bounds
	on the best-approximation error
	\begin{equation} \label{eq:best_Pdz}
		\inf_{\bv_h\in \bP\upd_{k,0}(\calT_h)} \|\bv-\bv_h\|_{\bL^p(\Dom)}.
	\end{equation}
	Let $\inter_{h0}\updav:\bL^1(\Dom)\to \bP\upd_{k,0}(\calT_h)$ 
	denote the quasi-interpolation operator with homogeneous boundary prescription
	associated with the Raviart--Thomas finite elements. Then, 
	\cite[Thm.~22.14]{Ern_Guermond_FEs_I_2021} shows again that 
	for all $r\in [0,\frac1p)$, there is $c$ such that
	for all $\bv\in \bW^{r,p}(\Dom)$, and 
	all $\hinH$, we have
	\begin{equation} \label{eq:inter_updavz}
		\|\bv-\inter_{h0}\updav(\bv)\|_{\bL^p(\Dom)} \le c\, h^{r}\ell_\Dom^{-r} \|\bv\|_{\bW^{r,p}(\Dom)}.
	\end{equation}
	% In the present work, we show that provided some (mild)
	% additional integrability assumption is made on $\DIV\bv$, the 
	% estimate~\eqref{eq:inter_updavz} can be fully localized to the mesh cells; see again
	% Theorem~\ref{th:loc_av_RT} and Corollary~\ref{cor:glob_av_RT}.
	
	We now show that the estimates \eqref{eq:inter_updav} and \eqref{eq:inter_updavz}
	can be localized under suitable additional assumptions on the divergence of
	the field $\bv$.
	Let us first observe that the domain of the normal trace operator can be
	extended to $\bY\upd(\Dom)\eqq \{\bv\in \Ldeuxd\st \DIV\bv\in L^q(\Dom)\}$
	for all $q\in (\frac{2d}{2+d},2]$. Indeed, for all $\bv\in \bY\upd(\Dom)$,
	$\gamma\upd(\bv)\in H^{-\frac12}(\front)$ can still be defined by duality 
	by setting for all $w\in H^{\frac12}(\front)$,
	\begin{equation} \label{eq:ext_traced_glob}
		\langle \gamma\upd(\bv),w\rangle_{\front} \eqq \int_\Dom \Big(\bv\SCAL\GRAD l(w)+
		(\DIV\bv)l(w)\Big)\dif x,
	\end{equation}
	where $l(w)$ denotes a lifting of $w$ in
	$H^1(\Dom)$. This is again a consequence of 
	$q> \frac{2d}{2+d}$ and H\"older's inequality.
	We can now state our main result. For simplicity, we estimate the
	quasi-interpolation error only in the $\bL^2$-norm.
	
	\begin{theorem}[Localized quasi-interpolation error estimate for Raviart--Thomas elements]
		\label{th:loc_av_RT}
		For all $r\in (0,1]$ and all $q\in (\frac{2d}{2+d},2]$, there is $c$ such that for
		all $\bv\in \bH^r(\Dom)$ with $\DIV\bv\in L^q(\Dom)$, all $K\in\calT_h$,
		and all $\hinH$, we have
		\begin{equation} \label{eq:loc_avd}
			\|\bv-\inter\updav_h(\bv)\|_{\bL^2(K)} \le
			c\sum_{K'\in\calT_K\upd} \Big\{ 
			h_{K'}^r |\bv|_{\bH^r(K')} + h_{K'}^{1+d(\frac12-\frac1q)} \|\DIV\bv\|_{L^q(K')}\Big\}.
		\end{equation}
		Moreover, assuming that $\gamma\upd(\bv)=0$, we also have
		\begin{equation} \label{eq:loc_avdz}
			\|\bv-\inter\updav_{h0}(\bv)\|_{\bL^2(K)} \le
			c\sum_{K'\in\calT_K\upd} \Big\{
			h_{K'}^r |\bv|_{\bH^r(K')} + h_{K'}^{1+d(\frac12-\frac1q)} \|\DIV\bv\|_{L^q(K')}\Big\}.
		\end{equation}
	\end{theorem}
	
	Using the same arguments as above, we readily infer the following result.
	
	\begin{corollary}[Localized best-approximation error for Raviart--Thomas elements]
		\label{cor:glob_av_RT}
		For all $r\in (0,1]$ and all $q\in (\frac{2d}{2+d},2]$, there is $c$ such that for
		all $\bv\in \bH^r(\Dom)$ with $\DIV\bv\in L^q(\Dom)$, 
		and all $\hinH$, we have
		\begin{equation} \label{eq:glob_avd}
			\inf_{\bv_h\in \bP\upd_k(\calT_h)}\|\bv-\bv_h\|_{\bL^2(\Dom)} \le
			c\,\bigg\{\sum_{K\in\calT_h} \Big\{ 
			h_{K}^{2r} |\bv|_{\bH^r(K)}^2 + h_{K}^{2+2d(\frac12-\frac1q)} \|\DIV\bv\|_{L^q(K)}^2\Big\}\bigg\}^{\frac12}.
		\end{equation}
		Moreover, assuming that $\gamma\upd(\bv)=0$, we also have
		\begin{equation} \label{eq:glob_avdz}
			\inf_{\bv_h\in \bP\upd_{k,0}(\calT_h)}\|\bv-\bv_h\|_{\bL^2(\Dom)} \le
			c\,\bigg\{\sum_{K\in\calT_h} \Big\{ 
			h_{K}^{2r} |\bv|_{\bH^r(K)}^2 + h_{K}^{2+2d(\frac12-\frac1q)} \|\DIV\bv\|_{L^q(K)}^2\Big\}\bigg\}^{\frac12}.
		\end{equation}
	\end{corollary}
	
	\begin{remark}[Comparison with canonical interpolation]
		The canonical Raviart--Thomas 
		interpolation operator and the quasi-interpolation operator achieve the same 
		convergence rates under the same 
		smoothness assumptions. Notice though that the estimates \eqref{eq:loc_avd}-\eqref{eq:loc_avdz}
		are localized over $\calT_K\upd$ whereas the estimate \eqref{loc_Ihd} is localized 
		over the single cell $K$. 
		% result as in Theorem~\ref{th:loc_av_RT}
		%(with an upper bound localized to a single mesh cell). 
		%Therefore, 
		The upper bounds in Corollary~\ref{cor:glob_av_RT} can also
		be derived  by invoking the canonical Raviart--Thomas 
		interpolation operator. 
		This is in contrast with the Nédélec elements where 
		the estimates~\eqref{eq:loc_avc}-\eqref{eq:loc_avcz} on the quasi-interpolation error 
		are essential to prove
		Corollary~\ref{cor:glob_av_Ned}. These estimates hold for the canonical Nédélec interpolation operator
		only if $r>\frac12$.
	\end{remark}

	\section{Application to Maxwell's equations} \label{sec:Maxwell}
	
	In this section, we briefly present two applications of the above
	results for the approximation of  Maxwell's equations using
	Nédélec finite elements.  We focus on simplified forms of Maxwell's
	equations obtained, \eg in the time-harmonic regime and in the
	eddy currents approximation.
	
	\subsection{Model problem}
	
	Given a source term $\bef\in \bL^q(\Dom)$ with $q\in (\frac{2d}{2+d},2]$,
	the model problem we consider consists of seeking 
	$\bA\in \Hrotz\eqq \{\bv\in \Hrot\tq \gamma\upc(\bv)=\bzero\}$ such that
	$a(\bA,\bb)=\ell(\bb)$ for all $\bb\in\Hrotz$, with the following 
	sesquilinear and antilinear forms:
	\begin{equation}
		a(\ba,\bb) \eqq \int_\Dom (\nu \ba \SCAL \bar{\bb} + 
		\kappa \ROT \ba \SCAL \ROT \bar{\bb}) \dif x,
		\qquad
		\ell(\bb) \eqq \int_\Dom  \bef \SCAL \bar{\bb} \dif x.
	\end{equation}
	We assume that the model parameters $\nu,\kappa$ are both bounded in 
	$\Dom$ and we set $\nu_\sharp\eqq \|\nu\|_{L^\infty(\Dom)}$,
	$\kappa_\sharp\eqq \|\kappa\|_{L^\infty(\Dom)}$. We also assume that
	there are real numbers $\theta$, $\nu_\flat>0$, $\kappa_\flat>0$ such that
	\begin{equation} \label{Hyp:mu_kappa_Weak_MHD_Maxwell}
		\essinf_{\bx\in\Dom}\Re\big(e^{i\theta}\nu(\bx)\big) \ge \nu_\flat\quad
		\text{and}
		\quad \essinf_{\bx\in\Dom}\Re\big(e^{i\theta}\kappa(\bx)\big) \ge \kappa_\flat.
	\end{equation}
	This condition ensures the coercivity of the sesquilinear form $a$.
	Therefore, the model problem is well-posed owing to the Lax--Milgram lemma,
	and its unique weak solution is such that $\nu \bA + \ROT(\kappa\ROT A)=\bef$ in
	$\Dom$ together with $\gamma\upc(\bA)=\bzero$. 
	To obtain a regularity result on the weak solution, 
	we assume that there is a partition $\{\Dom_m\}_{m\in\intset{1}{M}}$
	of $\Dom$ into $M\ge1$ disjoint polyhedral Lipschitz subsets such that 
	$\nu_{|\Dom_m}$ and $\kappa_{|\Dom_m}$ are constant for all $m\in\intset{1}{M}$.
	One can then show (see \cite{Jochmann:99,BoGuL:13}, see also
	\cite{BirSo:87,Costabel:90,AmBDG:98}) that there is $r\in (0,\frac12)$ such that
	\begin{equation} \label{eq:reg_sol_Max}
		\bA \in \bH^r(\Dom), \qquad \ROT\bA\in \bH^r(\Dom).
	\end{equation}
	
	\subsection{Classical approximation with prescribed boundary 
		conditions}
	
	In this setting, 
	the discrete problem approximated with Nédélec finite elements consists of seeking
	$\bA_h\in \bP\upc_{k,0}(\calT_h)$ such that $a(\bA_h,\bb_h)=\ell(\bb_h)$ for all
	$\bb_h\in \bP\upc_{k,0}(\calT_h)$. Recall that $k\in\polN$ denotes the 
	polynomial degree and that $\famTh$ is a shape-regular family of affine, matching, 
	simplicial meshes such that each mesh covers $\Dom$ exactly.
	Notice that the homogeneous boundary condition
	on the tangential component of $\bA_h$ at the boundary is explicitly enforced in
	the discrete problem. 
	
	The error analysis is performed by establishing suitable
	stability, consistency, and boundedness properties. 
	To avoid distracting technicalities,
	we do not track the dependency of the constants in
	the error analysis on the nondimensional factors $\frac{\nu_\sharp}{\nu_\flat}$, 
	$\frac{\kappa_\sharp}{\kappa_\flat}$, and $\nu_\sharp\kappa_\sharp^{-1}\ell_\Dom^2$
	(recall that $\ell_\Dom$ denotes
	a characteristic (global) length scale associated with $\Dom$). 
	Referring to \cite{ErnGu:18-CAMWA}
	(see also \cite[Chap.~44]{Ern_Guermond_FEs_II_2021}) for further insight, the main 
	error estimate states that there is $c$ such that we have
	\begin{equation}
		\|\bA-\bA_h\|_{\Hrot} \le c\, \inf_{\bb_h\in \bP\upc_{k,0}(\calT_h)}
		\|\bA-\bb_h\|_{\Hrot}.
	\end{equation}
	Moreover, invoking commuting quasi-interpolation operators, it is also shown therein
	that there is $c$ such that we have
	\begin{equation} \label{eq:quasi_opt_Max}
		\|\bA-\bA_h\|_{\Hrot} \le c\, \Big( \inf_{\bb_h\in \bP\upc_{k,0}(\calT_h)}
		\|\bA-\bb_h\|_{\Ldeuxd} + \ell_\Dom \inf_{\bd_h\in \bP\upd_{k,0}(\calT_h)}
		\|\ROT\bA-\bd_h\|_{\Ldeuxd}\Big).
	\end{equation}
	
	The analysis performed in this paper allows us to derive the following result where the 
	best approximation error is fully localized even when $r\in (0,\frac12)$.
	
	\begin{corollary}[Localized error estimate]
		There is $c$ such that we have
		\begin{equation}
			\|\bA-\bA_h\|_{\Hrot} \le c\,\bigg\{\sum_{K\in\calT_h} h_{K}^{2r} \Big(
			|\bA|_{\bH^r(K)}^2 + \ell_\Dom^{2(1-r)} \|\ROT\bA\|_{\bL^2(K)}^2
			+ \ell_\Dom^2 |\ROT\bA|_{\bH^r(K)}^2\Big) \bigg\}^{\frac12}.
		\end{equation}
	\end{corollary}
	
	\begin{proof}
		We combine the error estimate~\eqref{eq:quasi_opt_Max} with the 
		results from Corollary~\ref{cor:glob_av_Ned} and
		Corollary~\ref{cor:glob_av_RT}. In particular, we apply the estimate
		\eqref{eq:glob_avcz} with $\bv\eqq \bA$ (notice that indeed $\gamma\upc(\bA)=\bzero$)
		and $q\eqq2$, and the estimate \eqref{eq:glob_avdz} with $\bv\eqq\ROT\bA$
		(notice that indeed $\gamma\upd(\ROT\bA)=0$ and that $\DIV\bv=0$).
		This yields
		\[
		\|\bA-\bA_h\|_{\Hrot} \le c\,\bigg\{\sum_{K\in\calT_h} \Big(
		h_{K}^{2r} |\bA|_{\bH^r(K)}^2 + h_{K}^{2} \|\ROT\bA\|_{\bL^2(K)}^2
		+ \ell_\Dom^2 h_{K}^{2r} |\ROT\bA|_{\bH^r(K)}^2\Big) \bigg\}^{\frac12},
		\]
		whence the assertion follows from $h_{K}^{2}\le h_K^{2r}\ell_\Dom^{2(1-r)}$.
	\end{proof}

	\subsection{Nitsche's boundary penalty technique}
	
	Let us now consider a variant of the above discrete problem where the boundary
	condition is enforced by Nitsche's boundary penalty technique; see
	\cite{ErnGu:18-CMAM} and also \cite[Chap.~45]{Ern_Guermond_FEs_II_2021}.
	For all $K\in\calT_h$, we set $\nu_K\eqq |\nu_{|K}|$, $\kappa_K\eqq |\kappa_{|K}|$,
	$\nu_{r,K}\eqq \Re(e^{i\theta}\nu_K)$, and $\kappa_{r,K}\eqq \Re(e^{i\theta}\kappa_K)$.
	To avoid distracting technicalities,
	we do not track the dependency of the constants in
	the error analysis on the nondimensional factors
	$\frac{\nu_K}{\nu_{r,K}}$ and $\frac{\kappa_K}{\kappa_{r,K}}$.
	Recall that for every mesh boundary face $F\in\calFhb$, there is a mesh cell
	$K_-\in\calT_h$ such that $F=\partial K_-\cap \front$. 
	The discrete problem now consists of seeking $\bA_h\in \bP\upc_k(\calT_h)$
	such that $a_h(\bA_h,\bb_h)=\ell(\bb_h)$ for all $\bb_h\in \bP\upc_k(\calT_h)$,
	with $a_h(\SCAL,\SCAL)=a(\SCAL,\SCAL)-n_h(\SCAL,\SCAL)+s_h(\SCAL,\SCAL)$,
	with the following sesquilinear forms defined on $\bP\upc_k(\calT_h)\times 
	\bP\upc_k(\calT_h)$:
	\begin{equation}
		n_h(\ba_h,\bb_h)\eqq \sum_{F\in\calFhb}\int_F ((\kappa\ROT\ba_h)\times\bn_F)\SCAL
		\bar{\bb}_h \dif s,
		\qquad
		s_h(\ba_h,\bb_h)\eqq \sum_{F\in\calFhb} \eta_0 e^{-i\theta} \frac{\lambda_F}{h_F}
		\int_F (\ba_h\times\bn_F)\SCAL(\bar{\bb}_h\times\bn_F)\dif s, 
	\end{equation}
	where $\eta_0$ is a nondimensional parameter that must be chosen large enough
	so that the sesquilinear form $a_h$ is coercive on $\bP\upc_k(\calT_h)$
	(see \cite[Lem.~3.3]{ErnGu:18-CMAM} or \cite[Lem.~45.1]{Ern_Guermond_FEs_II_2021}). 
	
	To perform the error analysis, it is convenient
	to introduce the space $\bV\loS \eqq \{\ba\in\Hrotz\tq \kappa\ROT\ba \in \bV\upc(\Dom)\}$,
	where the space $\bV\upc(\Dom)$ is defined in \eqref{eq:def_Vc_Dom}.
	We observe that the exact solution satisfies $\bA\in\bV\loS$; indeed, $\bA\in\Hrotz$
	by construction, $\ROT(\kappa\ROT \bA)=\bef-\nu \bA \in \bL^q(\Dom)$ owing to
	our assumption on the source term $\bef$, and $\kappa\ROT\bA \in \bL^p(\Dom)$
	for some $p>2$ owing to the Sobolev embedding theorem (recall that $\ROT\bA\in
	\bH^r(\Dom)$ by~\eqref{eq:reg_sol_Max} and that $\kappa$ satisfies a suitable
	multiplier property as shown in \cite{Jochmann:99,BoGuL:13}). 
	Notice that the approximation error $\bA-\bA_h$ belongs to the space 
	$\bV_\sharp\eqq \bV\loS+\bP\upc_k(\calT_h)$ which we equip with the norm
	\begin{equation}
		\begin{aligned}
			\|\bb\|_{\bV_{\sharp}}^2 := & \sum_{K\in\calT_h}
			\Big\{ \nu_K \|\bb\|^2_{\bL^2(K)} + \kappa_K
			\|\ROT \bb\|^2_{\bL^2(K)}\Big\} 
			+ \sum_{F\in \calFhb} \frac{\kappa_{K_-}}{h_{K_-}}\|\bb_h\times \bn\|_{\bL^2(F)}^2 \\
			&+\sum_{F\in \calFhb} \kappa_{K_-}
			\Big\{h_{K_-}^{2d(\frac{1}{2}-\frac{1}{p})}\|\ROT \bb\|_{\bL^p(K_-)}^2 
			+ h_{K_-}^{2+2d(\frac{1}{2}-\frac{1}{q})}\|\ROT (\ROT \bb)\|_{\bL^q(K_-)}^2 
			\Big\},
		\end{aligned}
	\end{equation}
	for all $\bb\eqq \bb\loS+\bb_h\in \bV_\sharp$ with $\bb\loS\in \bV\loS$ and
	$\bb_h\in \bP\upc_k(\calT_h)$.
	The following error estimate is established in \cite[Sec.~7.2]{ErnGu:18-CMAM} and
	\cite[Thm.~45.6]{Ern_Guermond_FEs_II_2021} (with $q=2$): There is $c$ such that we have
	\begin{equation}
		\|\bA-\bA_h\|_{\bV_\sharp} \le c\, \inf_{\bb_h\in \bP\upc_{k}(\calT_h)}
		\|\bA-\bb_h\|_{\bV_\sharp}.
	\end{equation}
	
	The analysis performed above allows us to derive the following error estimate where the 
	best approximation error is fully localized even when $r\in (0,\frac12)$.
	
	\begin{corollary}[Localized error estimate]
		There is $c$ such that we have
		\begin{align}
			\|\bA-\bA_h\|_{\bV_\sharp} \le {}& c\,\kappa_\sharp^{\frac12} 
			\bigg\{\sum_{K\in\calT_h} h_{K}^{2r} \Big(
			\ell_\Dom^{-2} |\bA|_{\bH^r(K)}^2 + \ell_\Dom^{-2r} \|\ROT\bA\|_{\bL^2(K)}^2
			+ |\ROT\bA|_{\bH^r(K)}^2\Big)  \nonumber \\
			& +   \sum_{F\in\calFhb} 
			h_{K_-}^{2d(\frac{d+2}{2d}-\frac1q)} \|\bef-\nu\bA\|_{\bL^q(K_-)}^2\bigg\}^{\frac12}.
		\end{align}
	\end{corollary}
	
	\begin{proof}
		Adapting the arguments from the proof of \cite[Thm.~45.6]{Ern_Guermond_FEs_II_2021}
		(setting $t=r$ therein) where the infimum is realized by using a commuting 
		quasi-interpolant with prescribed boundary conditions, we infer that 
		\begin{align*}
			\inf_{\bb_h\in \bP\upc_{k}(\calT_h)} \|\bA-\bb_h\|_{\bV_\sharp}
			\le {}& c\, \bigg\{ \nu_\sharp
			\inf_{\bb_h\in \bP\upc_{k,0}(\calT_h)}
			\|\bA-\bb_h\|_{\Ldeuxd}^2 + \kappa_\sharp \inf_{\bd_h\in \bP\upd_{k,0}(\calT_h)}
			\|\ROT\bA-\bd_h\|_{\Ldeuxd}^2 \\
			& + \kappa_\sharp \sum_{F\in\calFhb} \Big( h_{K_-}^{2r} |\ROT\bA|_{\bH^r(K_-)}^2 
			+ h_{K_-}^{2d(\frac{d+2}{2d}-\frac1q)} \|\bef-\nu\bA\|_{\bL^q(K_-)}^2 \Big)\bigg\}^{\frac12}.
		\end{align*}
		Combining this estimate with the estimates
		\eqref{eq:glob_avcz} and \eqref{eq:glob_avdz} proves the assertion (as above, we hide
		the nondimensional factor $\nu_\sharp\kappa_\sharp^{-1}\ell_\Dom^2$ in the generic 
		constant $c$).
	\end{proof}

	\bibliographystyle{siam}
	\bibliography{refs}
	
\end{document}